\documentclass[12pt]{amsart}
\usepackage{fullpage}
\usepackage[pagebackref]{hyperref}
\usepackage[backrefs]{amsrefs}

\newtheorem{definition}{Definition}[section]
\newtheorem{thm}{Theorem}[section]
\newtheorem{prop}{Proposition}[section]
\newtheorem{lemma}{Lemma}[section]
\newtheorem{cor}{Corollary}[section]
\newtheorem{rmk}{Remark}[section]
\newtheorem*{theorem}{Theorem}

\numberwithin{equation}{section}

\title{On the adjoint representation of $\mathfrak{sl}_n$ and \\ the Fibonacci numbers}
\date{\today}
\author{ \textsc{Pamela E. Harris}}

\begin{document}
\maketitle
\begin{center}
University of Wisconsin - Milwaukee \\ Department of Mathematical Sciences \\ P.O. Box 0413, Milwaukee WI 53201\\
peharris@uwm.edu

\end{center}

\begin{abstract}
We decompose the adjoint representation of $\mathfrak{sl}_{r+1}=\mathfrak {sl}_{r+1}(\mathbb C)$ by a purely combinatorial approach based on the introduction of a certain subset of the Weyl group called the \emph{Weyl alternation set} associated to a pair of dominant integral weights.  The cardinality of the Weyl alternation set associated to the highest root and zero weight of $\mathfrak {sl}_{r+1}$ is given by the $r^{th}$ Fibonacci number. We then obtain the exponents of $\mathfrak {sl}_{r+1}$ from this point of view.
\end{abstract}

\section{Introduction}
\bigskip
Let $G$ be a simple linear algebraic group over $\mathbb C$, and $T$ a maximal algebraic torus in $G$ of dimension $r$. Let $B$, $T\subseteq B \subseteq G$, be a choice of Borel subgroup. Then let  $\mathfrak g$, $\mathfrak h$, and $\mathfrak b$ denote the Lie algebras of $G$, $T$, and $B$ respectively. Let $\Phi$ be the set of roots corresponding to $(\mathfrak {g,h})$, and let $\Phi^+\subseteq\Phi$ be the choice of positive roots with respect to $\mathfrak b$. Let $P(\mathfrak g)$ be the integral weights with respect to $\mathfrak h$, and let $P_+(\mathfrak g)$ be the dominant integral weights. The theorem of the highest weight asserts that any finite dimensional complex irreducible representation of $\mathfrak g$ is equivalent to a highest weight representation with dominant integral highest weight $\lambda$, denoted by $L(\lambda)$. A good general reference for highest weight theory in the finite dimensional setting, and for our terminology and notation is~\cite{GW}.

Let $W=Norm_G(T)/T$ denote the Weyl group corresponding to $G$ and $T$. For $w\in W$, we let $\ell(w)$ denote the length of $w$. Set $\epsilon(w)=(-1)^{\ell(w)}$. Kostant's partition function is the non-negative integer valued function, $\wp$, defined on $\mathfrak h^*$ by $\wp(\xi)$ = number of ways $\xi$ may be written as a non-negative integral sum of positive roots, for $\xi\in\mathfrak h^*$.

An area of interest in combinatorial representation theory is finding the multiplicity of a weight $\mu$ in $L(\lambda)$. One way to compute this multiplicity, denoted $m(\lambda,\mu)$, is by Kostant's weight multiplicity formula~\cite{KMF}:
\begin{align}
m(\lambda,\mu)=\displaystyle\sum_{\sigma\in W}^{}\epsilon(\sigma)\wp(\sigma(\lambda+\rho)-(\mu+\rho))\label{mult formula},
\end{align} where $\rho=\frac{1}{2}\sum_{\alpha\in\Phi^+}\alpha$. One complication in using (\ref{mult formula}) to compute multiplicities is that closed formulas for the value of Kostant's partition function are not known in much generality. A second complication concerns the exponential growth of the Weyl group order as $r\rightarrow\infty$. In practice, most terms in (\ref{mult formula}) are zero and hence do not contribute to the overall multiplicity. With the aim of describing the contributing terms in Kostant's weight multiplicity formula, we give the following definition.
\begin{definition} For $\lambda,\mu$ dominant integral weights of $\mathfrak g$ define the Weyl alternation set to be \begin{center}$\mathcal A(\lambda,\mu)=\{\sigma\in W|\;\wp(\sigma(\lambda+\rho)-(\mu+\rho))>0\}$.\end{center}
\end{definition}

Let $\{\varpi_1,\ldots,\varpi_r\}$ be the set of fundamental weights of $\mathfrak g$. Let $(,):\mathfrak h^*\times\mathfrak h^*\rightarrow\mathbb C$ be the symmetric non-degenerate form corresponding to the trace form. Then the following proposition will be useful in determining Weyl alternation sets. Its proof is an easy exercise.

\begin{prop}\label{prop1} Let $\xi\in\mathfrak h^*$. Then $\wp(\xi)>0$ if and only if\footnote{$\mathbb N=\{0,1,2,\ldots\}$.} $(\varpi_i,\xi)\in\mathbb N$, for all $1\le i\le r$.
\end{prop}

The purpose of this short note is to demonstrate that the sets $\mathcal A(\lambda,\mu)$ are combinatorially interesting. To this end, we specialize to the case when $\mathfrak g=\mathfrak {sl}_{r+1}$ and in Section 2 prove the following:

\begin{theorem} If $r\ge 1$ and $\tilde\alpha$ is the highest root of $\mathfrak {sl}_{r+1}$, then $|\mathcal A(\tilde\alpha,0)|=F_r$, where $F_r$ denotes the $r^{th}$ Fibonacci number.
\end{theorem}

This result gives rise to a (new) combinatorial identity associated to a Cartan subalgebra of $\mathfrak {sl}_{r+1}$, which we present in Section 3. The non-zero weights, $\mu$, of $\mathfrak {sl}_{r+1}$ are considered in Section 4 from the same point of view. We now introduce notation and terminology to make our approach precise.

\begin{rmk}The above theorem does not generalize to other simple Lie algebras, which motivates further research.
\end{rmk}

\section{The zero weight space}

Let $r\ge 1$, and let $n=r+1$. Let $G=SL_n(\mathbb C)$, $\mathfrak g=\mathfrak{sl}_n(\mathbb C)$, and let\footnote{$diag[a_1,\ldots,a_n]$ is the diagonal $n\times n$ matrix whose entries are $a_1,\ldots,a_n$.} \begin{center}$\mathfrak h=\{diag[a_1,\ldots,a_n]|a_1,\ldots,a_n\in\mathbb C,\displaystyle\sum_{i=1}^{n}a_i=0\}$\end{center} be a fixed choice of Cartan subalgebra. Let $\mathfrak b$ denote the set of $n\times n$ upper triangular complex matrices with trace zero. For $1\le i\le n$, define the linear functionals $\varepsilon_i:\mathfrak h\rightarrow \mathbb C$ by $\varepsilon_i(H)=a_i$, for any $H=diag[a_1,\ldots,a_n]\in\mathfrak h$. The Weyl group, $W$, is isomorphic to $S_n$, the symmetric group on $n$ letters, and acts on $\mathfrak h^*$ by permutations of $\varepsilon_1,\ldots,\varepsilon_n$.

For each $1\le i\le r$, let $\alpha_i=\varepsilon_i-\varepsilon_{i+1}$. Then the set of simple and positive roots corresponding to $(\mathfrak g,\mathfrak b)$ are $\Delta=\{\alpha_1,\ldots,\alpha_{r}\}$, and $\Phi^+=\{\varepsilon_i-\varepsilon_j|\;1\le i<j\le n\}$ respectively. The highest root is $\tilde\alpha=\varepsilon_1-\varepsilon_n=\alpha_1+\cdots+\alpha_{r}$. The fundamental weights are defined by $\varpi_i=\varepsilon_1+\cdots+\varepsilon_i-\frac{i}{n}(\varepsilon_1+\cdots+\varepsilon_n)$, where $1\le i\le r$. Then the sets of integral weights, and dominant integral weights are
\begin{align*}P(\mathfrak g)&=\{a_1\varpi_1+\cdots+a_{r}\varpi_{r}|a_i\in\mathbb Z\text{, for all } i=1,\ldots,r\},\text{ and}\\
P_+(\mathfrak g)&=\{a_1\varpi_1+\cdots+a_{r}\varpi_{r}|a_i\in\mathbb N\text{, for all } i=1,\ldots,r\} \text{ respectively.}
\end{align*}

Let $(,)$ be the symmetric bilinear form on $\mathfrak h^*$ corresponding to the trace form as in ~\cite{GW}. Observe that if $H=diag[a_1,\ldots,a_n]\in\mathfrak h$, then $(\varepsilon_1+\cdots+\varepsilon_n)(H)=a_1+\cdots+a_n=0$. Thus for any $1\le i\le r$, we can write $\varpi_i=\varepsilon_1+\cdots+\varepsilon_i$.

As a simplification we will write $\xi\in\mathfrak h^*$ as an $n$-tuple whose $i$-th coordinate is given by $(\varepsilon_i,\xi)$. Hence if $\xi=(\xi_1,\ldots,\xi_n)\in\mathfrak h^*$ and $1\le i\le r$, then $(\varpi_i,\xi)=\xi_1+\cdots+\xi_i$. Also notice $\tilde\alpha=(1,\;0,\;0,\ldots,\;0,\;-1)$ is the highest root, $\rho=(n-1,\;n-2\;,n-3,\ldots,\;2,\;1,\;0)$, and $\tilde\alpha+\rho=(n,\;n-2,\;n-3,\ldots,\;2,\;1,\;-1)$.
It will be useful to relabel $\tilde\alpha+\rho=(a_1,a_2,\ldots,\;a_n)$, where
\begin{align}
a_j=\begin{cases}n & \text{if $j=1$}\\
                -1 & \text{if $j=n$}\\
                n-j&\text{otherwise.}\end{cases}\label{relabel}
\end{align}
Then for any $\sigma\in W$,
\begin{align}\sigma(\tilde\alpha+\rho)-\rho &= (a_{\sigma^{-1}(1)}-n+1,\;a_{\sigma^{-1}(2)}-n+2,\ldots,\;a_{\sigma^{-1}(n-1)}-1,\;a_{\sigma^{-1} (n)})\label{1.3}.
\end{align}

\begin{thm}\label{t1}
Let $\sigma\in S_n$. Then $\sigma\in\mathcal A(\tilde\alpha,0)$ if and only if $\sigma(1)=1$, $\sigma(n)=n$, and $|\sigma(i)-i|\le 1$, for all $1\le i\le n$.
\end{thm}

We begin with the following technical propositions. The proof of Proposition \ref{product} is an easy exercise.
\begin{prop}\label{product} Let $\sigma\in S_n$. Then $\sigma$ is a product of commuting neighboring transpositions if and only if $|\sigma(i)-i|\le 1$, for all $1\le i\le n$.
\end{prop}


\begin{prop}\label{p1}
Let $\sigma\in S_n$ such that $|\sigma(i)-i|\le 1$, for all $1\le i\le n$. Then $\sigma(1)=1$, and $\sigma(n)=n$ if and only if for any $ 1\le i\le r$,
\begin{align*}(\varpi_i,\sigma(\tilde\alpha+\rho)-\rho)=\begin{cases}
                1 & \text{if $\{\sigma(1),\sigma(2)\ldots,\sigma(i)\}=\{1,2,\ldots,i\}$}\\
                                            0 & \text{if $\sigma(i)=i+1$.}
                                            \end{cases}\end{align*}
\end{prop}

\begin{proof}
$(\Rightarrow)$ Assume $\sigma \in S_n$ such that $\sigma(1)=1$, $\sigma(n)=n$, and $| \sigma(i)-i|\le 1$, for all \linebreak$1\le i\le n$. By Proposition \ref{product}, $\sigma$ is a product of commuting neighboring transpositions, and hence $\sigma=\sigma^{-1}$. Then $\sigma(\tilde\alpha+\rho)-\rho=(1,2-\sigma(2),\ldots,(n-1)-\sigma(n-1),-1)$.

We proceed by induction on $i$. If $i=1$, then $\sigma(1)=1$ and $(\varpi_1,\sigma(\tilde\alpha+\rho)-\rho)=1$. Now let $1< i\le r$, and assume that for any $j\le i-1$, \begin{center}$(\varpi_j,\sigma(\tilde\alpha+\rho)-\rho)=\begin{cases} 1 & \text{if $\{\sigma(1),\sigma(2)\ldots,\sigma(j)\}=\{1,2,\ldots,j\}$}\\
                                            0 & \text{if $\sigma(j)=j+1$.}
                                            \end{cases}$\end{center}
Suppose that $j=i$. By induction hypothesis, since $j-1\le i-1$, we have that
\begin{center}$(\varpi_{j-1},\sigma(\tilde\alpha+\rho)-\rho)=\begin{cases} 1 &\text{if $\{\sigma(1),\;\sigma(2),\ldots,\sigma(j-1)\}=\{1,\;2,\ldots,j-1\}$}\\
0 &\text{if $\sigma(j-1)=j$.}
\end{cases}$\end{center}

Case 1: Assume $(\varpi_{j-1},\sigma(\tilde\alpha+\rho)-\rho)=1$. So $\{\sigma(1),\ldots,\sigma(j-1)\}=\{1,\ldots,j-1\}$, and $\sigma(j)=j$ or $\sigma(j)=j+1$. Hence,
\begin{align*}(\varpi_j,\sigma(\tilde\alpha+\rho)-\rho)&=1+(j-\sigma(j))= \begin{cases}
1 & \text{if $\{\sigma(1),\ldots,\sigma(j)\}=\{1,\ldots,j\}$}\\
0 & \text{if $\sigma(j)=j+1$.}\end{cases}
\end{align*}

Case $2$: Assume $(\varpi_{j-1},\sigma(\tilde\alpha+\rho)-\rho)=0$. So $\sigma(j-1)=j$, and observe that  \begin{align*}(\varpi_{j-1},\sigma(\tilde\alpha+\rho)-\rho)&=(\varpi_{j-2},\sigma(\tilde\alpha+\rho)-\rho)+((j-1)-\sigma(j-1))\\
&=(\varpi_{j-2},\sigma(\tilde\alpha+\rho)-\rho)-1=0.
\end{align*} Hence $(\varpi_{j-2},\sigma(\tilde\alpha+\rho)-\rho)=1$, and by induction hypothesis we have that \\$\{\sigma(1),\sigma(2),\ldots,\sigma(j-2)\}=\{1,2,\ldots,j-2\}$. So $\sigma(j)=j+1$ or $\sigma(j)=j-1$. If $\sigma(j)=j+1$, then $\sigma(k)=j-1$, for some integer $k\ge j+1$. This implies that $|\sigma(k)-k|\ge 2,$ a contradiction. Thus $\sigma(j)=j-1$, $\{\sigma(1),\ldots,\sigma(j)\}=\{1,\ldots,j\}$, and $(\varpi_j,\sigma(\tilde\alpha+\rho)-\rho)=j-\sigma(j)= 1$.

$(\Leftarrow)$ Let $\sigma \in S_n$ such that $| \sigma(i)-i|\le 1$, for all $ 1 \le i \le n$. Suppose that\\ $(\varpi_j,\sigma(\tilde\alpha+\rho)-\rho)=\begin{cases}
                1 & \text{if $\{\sigma(1),\sigma(2)\ldots,\sigma(j)\}=\{1,2,\ldots,j\}$}\\
                                            0 & \text{if $\sigma(j)=j+1$,}
                                            \end{cases}$ holds for any $1\le j \le r$.
Proposition \ref{product} implies that $\sigma=\sigma^{-1}$, hence (\ref{1.3}) simplifies to \begin{center}$\sigma(\tilde\alpha+\rho)-\rho=(a_{\sigma(1)}-n+1,\;a_{\sigma(2)}-n+2,\ldots,\;a_{\sigma(n-1)}-1,\;a_{\sigma (n)}),$\end{center} where $a_i$ is defined by (\ref{relabel}).
If $\sigma(1)\ne 1$, then $\sigma(1)=2$, and $(\varpi_1,\sigma(\tilde\alpha+\rho)-\rho)=-1$, a contradiction. Thus $\sigma(1)=1$. If $\sigma(n)\ne n$, then $\sigma(n)=n-1$, and $(\varpi_{r},\sigma(\tilde\alpha+\rho)-\rho)=-a_{\sigma(n)}=-1$, another contradiction. Thus $\sigma(n)=n$.
\end{proof}

\begin{proof} [Proof of Theorem \ref{t1}]

Recall that $\sigma\in\mathcal A(\tilde\alpha,0)$ if and only if $\wp (\sigma(\tilde\alpha +\rho)-\rho)>0$. Hence it suffices to show that $\wp (\sigma(\tilde\alpha +\rho)-\rho)>0$ if and only if $\sigma(1)=1$, $\sigma(n)=n$, and $|\sigma(i)-i|\le 1$, for all $1\le i\le n$.

$(\Rightarrow)$ Let $\sigma\in S_n$ such that $\wp(\sigma(\tilde\alpha+\rho)-\rho)>0$. So, by Proposition \ref{prop1}, $(\varpi_i,\sigma(\tilde\alpha+\rho)-\rho)\in\mathbb N$, for all $1\le i\le r$. By (\ref{1.3}) we have that \begin{center}$\sigma(\tilde\alpha+\rho)-\rho=(a_{\sigma^{-1}(1)}-n+1,\;a_{\sigma^{-1}(2)}-n+2,\ldots,\;a_{\sigma^{-1}(n-1)}-1,\;a_{\sigma^{-1} (n)})$,\end{center} where $a_i$ is defined by (\ref{relabel}). We want to prove $\sigma(1)=1$, $\sigma(n)=n$, and $\lvert i-\sigma (i)\lvert\le 1$, for all $1\le i\le n$.

If $1<\sigma^{-1}(1)\le n$, then $(\varpi_1,\sigma(\tilde\alpha+\rho)-\rho)=a_{\sigma^{-1}(1)}-n+1<0$, a contradiction. So $\sigma^{-1}(1)=1$, and hence $\sigma(1)=1$. If $1< \sigma^{-1}(n)<n$, then $(\varpi_{r},\sigma(\tilde\alpha+\rho)-\rho)
=-a_{\sigma^{-1}(n)}=\sigma^{-1}(n)-n<0$, a contradiction.
So $\sigma^{-1}(n)=n$, and hence $\sigma(n)=n$.

Hence (\ref{1.3}) simplifies to \begin{center}$\sigma(\tilde\alpha+\rho)-\rho=(1,\;2-\sigma^{-1}(2),\;3-\sigma^{-1}(3),\ldots,\;(n-1)-\sigma^{-1}(n-1),\;-1)$.\end{center}

Observe that if $|i-\sigma^{-1}(i)|\le 1$, for all $1<i<n$, then Proposition \ref{product} implies $\sigma^{-1}=\sigma$ and thus $|i-\sigma(i)|\le 1$, for all $1<i<n$. Thus it suffices to show that $|i-\sigma^{-1}(i)|\le 1$, for all $1<i<n$.

We proceed by induction on $i$. If $i=2$, then $(\varpi_2,\sigma(\tilde\alpha+\rho)-\rho)=1+2-\sigma^{-1}(2)\ge 0$ if and only if $\sigma^{-1}(2)\le 3$. Since $\sigma^{-1}(1)=1$, we have that $\sigma^{-1}(2)=2$ or $\sigma^{-1}(2)=3$ and in either case  $|2-\sigma^{-1}(2)|\le 1$.

Now let $2\le i\le n-1$ and assume that $|j-\sigma^{-1}(j)|\le 1$ holds for any $j<i$.

Suppose that $j=i$. Since $j-1<i$, we have that $|(j-1)-\sigma^{-1}(j-1)|\le 1$. Thus $\sigma^{-1}(j-1)=j-2,\;j-1$ or $j$. If $\sigma^{-1}(j-1)=j-2$ or $j-1$, then $\{\sigma^{-1}(1),\ldots,\sigma^{-1}(j-1)\}=\{1,\ldots,j-1\}$ and $(\varpi_{j-1},\sigma(\tilde\alpha+\rho)-\rho)=1$. Hence $(\varpi_j,\sigma(\tilde\alpha+\rho)-\rho)=1+j-\sigma^{-1}(j)\ge 0$ if and only if $\sigma^{-1}(j)\le j+1$. Thus $\sigma^{-1}(j)=j$ or $j+1$, and in either case $|j-\sigma^{-1}(j)|\le 1$.

Now suppose that $\sigma^{-1}(j-1)=j$. Since $j-2<i$, we have that $|(j-2)-\sigma^{-1}(j-2)|\le 1$. Hence $\{\sigma^{-1}(1),\ldots,\sigma^{-1}(j-2)\}=\{1,\ldots, j-2\}$ and $(\varpi_{j-1},\sigma(\tilde\alpha+\rho)-\rho)=0$. Then $(\varpi_{j},\sigma(\tilde\alpha+\rho)-\rho)=j-\sigma^{-1}(j)\ge 0$ if and only if $\sigma^{-1}(j)\le j$. Thus $\sigma^{-1}(j)=j-1$ and $|j-\sigma^{-1}(j)|\le 1$, which completes our induction step.

$(\Leftarrow)$ Let $\sigma\in S_n$ such that $\sigma(1)=1$, $\sigma(n)=n$, and $|\sigma(i)-i|\le 1$, for all $1\le i \le n$. Proposition \ref{p1} implies that $(\varpi_i,\sigma(\tilde\alpha+\rho)-\rho)= 0$ or $1$, for all $1\le i\le r$. Therefore, by Proposition \ref{prop1}, $\wp (\sigma(\tilde\alpha+\rho)-\rho)>0$.
\end{proof}

\begin{definition} The Fibonacci numbers are the sequence of numbers, $\{F_n\}_{n=1}^{\infty}$, defined by the recurrence relation \begin{align*}F_n=F_{n-1}+F_{n-2} \text{ for } n\ge 3, \text{ and } F_1=F_2=1.\end{align*}
\end{definition}
\begin{rmk} The Fibonacci numbers are well known for their prevalence throughout mathematics. We refer the reader to \cite{Fibonacci}.
\end{rmk}
We leave the proof of the following lemma to the reader.
\begin{lemma}\label{p2}
If $m\ge 1$, then $|\{\sigma\in S_m:|\sigma(i)-i|\le 1\text{, for all }1\le i\le m\}|=F_{m+1}$.
\end{lemma}
Now for the main result of this section.
\begin{thm}\label{fib} If $r\ge 1$ and $\tilde\alpha$ is the highest root of $\mathfrak {sl}_{r+1}$, then $|\mathcal A(\tilde\alpha,0)|=F_r$.
\end{thm}
\begin{proof} By Theorem \ref{t1} we know that \begin{center}$\mathcal A(\tilde\alpha,0)=\{\sigma\in S_{r+1}|\;\sigma(1)=1,\sigma(r+1)=r+1\text{, and }|\sigma(i)-i|\le 1\text{, $\forall$ }1\le i\le r+1\}$.\end{center} Notice that the sets $\mathcal A(\tilde\alpha,0)$ and $\{\sigma\in S_{r-1}:\;|\sigma(i)-i|\le 1\text{, for all }1\le i\le r-1\}$ have the same cardinality. Therefore, by Lemma \ref{p2}, $|\mathcal A(\tilde\alpha,0)|=F_r$.
\end{proof}

\begin{rmk}
For $1\le i\le r$, let $s_i$ denote the simple root reflection corresponding to $\alpha_i\in\Delta$. Then $s_{i}(\varepsilon_k)=\varepsilon_{\sigma(k)}$, where $\sigma$ is the neighboring transposition $(i\;\;i+1)\in S_n$. Notice $s_i s_j=s_j s_i$ if and only if $i$ and $j$ are non-consecutive integers between $1$ and $r$. Then the following corollary describes the elements of $\mathcal A(\tilde\alpha,0)$ as products of commuting simple root reflections.
\end{rmk}

\begin{cor}\label{reflections} Let $\sigma\in S_n$. Then $\sigma\in\mathcal A(\tilde\alpha,0)$ if and only if $\sigma=s_{{i_1}}s_{{i_2}}\cdots s_{{i_k}}$, for some non-consecutive integers $2\le i_1,\ldots,i_k\le r-1$.
\end{cor}
\begin{proof}The corollary follows from Theorem \ref{t1}, and Proposition \ref{product}.\end{proof}

\section{A $q$-analog}

The $q$-analog of Kostant's partition function is the polynomial valued function, $\wp_q$, defined on $\mathfrak h^*$ by
\begin{align*}
\wp_q(\xi)=c_0+c_1q+\cdots+c_k q^k,
\end{align*} where $c_j$= number of ways to write $\xi$ as a non-negative integral sum of exactly $j$ positive roots, for $\xi\in\mathfrak h^*$.
Lusztig introduced the $q$-analog of Kostant's weight multiplicity formula by defining a polynomial, which when evaluated at $1$ gives the multiplicity of the dominant weight $\mu$ in the irreducible module $L(\lambda)$~\cite{LL}. This formula is given by
\begin{align*} m_q(\lambda,\mu)=\sum\limits_{\sigma\in W}^{}\epsilon(\sigma)\wp_q(\sigma(\lambda+\rho)-(\mu+\rho)).\end{align*}

Let $\mathfrak g$ be a simple Lie algebra of rank $r$. In the case when $\tilde\alpha$ is the highest root of $\mathfrak g$, it is known that $m_q(\tilde\alpha,0)=\sum\limits_{i=1}^{r} q^{e_i}$, where $e_1,\ldots,e_r$ are the exponents of $\mathfrak g$~\cite{Kostant}. For $r\ge 1$, the exponents of $\mathfrak {sl}_{r+1}$ are $1,2,\ldots,r$~\cite{H}. In this section we use the Weyl alternation set $\mathcal A(\tilde\alpha,0)$ to give a combinatorial proof of:

\begin{thm}\label{t2}
If $\tilde\alpha$ is the highest root of $\mathfrak {sl}_{r+1}$, then $m_q(\tilde\alpha,0)=q+q^2+\cdots+q^r$.
\end{thm}
\begin{cor}\label{c1}
If $\tilde\alpha$ is the highest root of $\mathfrak {sl}_{r+1}$, then $m(\tilde\alpha,0)=r$.
\end{cor}
\begin{proof} This follows from Theorem \ref{t2} and the fact that $m_q(\tilde\alpha,0)|_{q=1}=m(\tilde\alpha,0)$.
\end{proof}

Throughout the remainder of this section let $r\ge 1$, and let $\tilde\alpha$ denote the highest root of $\mathfrak {sl}_{r+1}$. We leave the proofs of Lemmas \ref{p3} and \ref{p5} to the reader.

\begin{lemma}\label{p3}
Let $\sigma=s_1 s_2\cdots s_k\in\mathcal A(\tilde\alpha,0)$, where $i_1,\ldots,i_{k}$ are non-consecutive integers between $2$ and $r-1$. Then $\sigma(\tilde\alpha+\rho)-\rho=\tilde\alpha-\sum_{j=1}^k \alpha_{i_j}$.
\end{lemma}


\begin{lemma}\label{p5}
The cardinality of the set $\{\sigma\in\mathcal A(\tilde\alpha,0)\;|\;\ell(\sigma)=k\}$ is $\binom{r-1-k}{k}$, and \\ $max\{\;\ell(\sigma)\;|\;\sigma\in\mathcal A(\tilde\alpha,0)\}=\lfloor \frac{r-1}{2}\rfloor$.
\end{lemma}
We now prove the following combinatorial identity:
\begin{prop}\label{p6}
If $\sigma\in\mathcal A(\tilde\alpha,0)$, then $\wp_q (\sigma(\tilde\alpha+\rho)-\rho)=q^{1+\ell(\sigma)}(1+q)^{r-1-2\ell(\sigma)}$.
\end{prop}

\begin{proof}
If $\sigma\in\mathcal A(\tilde\alpha,0)$ with $\ell(\sigma)=0$, then $\sigma=1$ and $\sigma(\tilde\alpha+\rho)-\rho=\tilde\alpha=\alpha_1+\cdots+\alpha_r$. Since $\Phi^+=\{\alpha_i:1\le i\le r\}\cup \{\alpha_i+\cdots+\alpha_j:1\le i<j\le r\}$, for any $i\ge 0$, we can think of $c_{i+1}$, the coefficient of $q^{i+1}$ in $\wp_q(\alpha_1+\cdots+\alpha_r)$, as the number of ways to place $i$ lines in $r-1$ slots. Hence $c_{i+1}=\binom{r-1}{i}$ and $\wp_q(\tilde\alpha)=\sum_{i=0}^{r-1}\binom{r-1}{i}q^{i+1}=q(1+q)^{r-1}$.

If $\sigma\in\mathcal A(\tilde\alpha,0)$ with $\ell(\sigma)=k\neq 0$, then Corollary \ref{reflections} implies that $\sigma=s_1 s_2\cdots s_k$, for some non-consecutive integers $2\le i_1,i_2,\ldots,i_{k}\le r-1$. Then by Lemma \ref{p3}, $\sigma(\tilde\alpha+\rho)-\rho =\tilde\alpha-\sum_{j=1}^{k}\alpha_{i_j}$. Let $c_j$ denote the coefficient of $q^j$ in $\wp_q(\sigma(\tilde\alpha+\rho)-\rho)$. Since $\sigma$ subtracts $k$ many non-consecutive simple roots from $\tilde\alpha$, we will at a minimum need $k+1$ positive roots to write $\tilde\alpha-\sum_{j=1}^{k}\alpha_{i_j}$. So $c_j=0$, whenever $j< k+1$. Also observe that  $\tilde\alpha-\sum_{j=1}^{k}\alpha_{i_j}$ can be written with at most $r-k$ positive roots. Hence $c_j=0$, whenever $j> n-k$.

For $i\ge 0$, we can think of $c_{k+1+i}$ as the number of ways to place $i$ lines in $r-1-2k$ slots. This is because for each simple root that $\sigma$ removes from $\tilde\alpha$, we lose 2 slots in which to place a line, one before and one after. So $c_{k+1+i}=\binom{r-1-2k}{i}$, whenever $0\le i\le r-1-2k$.

Therefore $\wp_q(\sigma(\tilde\alpha+\rho)-\rho)=\sum\limits_{i=0}^{r-1-2k}\binom{r-1-2k}{i}q^{k+1+i}=q^{1+k}(1+q)^{r-1-2k}$.
\end{proof}
We obtain the following closed formula of Kostant's partition function by setting $q=1$ in Proposition \ref{p6}.

\begin{lemma}\label{lemma}
If $\sigma\in\mathcal A(\tilde\alpha,0)$, then $\wp(\sigma(\tilde\alpha+\rho)-\rho)=2^{r-1-2\ell(\sigma)}$.
\end{lemma}

The following proposition will be used in the proof of Theorem \ref{t2}.
\begin{prop}\label{p7}
For $r\ge 1$, $\displaystyle\sum_{k=0}^{\lfloor \frac{r-1}{2}\rfloor}(-1)^k \binom{r-1-k}{k} q^{1+k}(1+q)^{r-1-2k}=\displaystyle\sum_{i=1}^{r}q^i$.
\end{prop}

\begin{proof}
Equation (4.3.7) in \cite{Wilf} shows that for integers $k$ and $n\ge 0,$
\begin{align}
\displaystyle\sum_{k\le \frac{n}{2}}^{}(-1)^k \binom{n-k}{k} q^{k}(1+q)^{n-2k}=\frac{1-q^{n+1}}{1-q}.\label{wilfeq}
\end{align}
Suppose $r\ge 1$, and let $n=r-1\ge 0$. Then by (\ref{wilfeq}) we have that
\begin{align*}
\displaystyle\sum_{k=0}^{\lfloor \frac{r-1}{2}\rfloor}(-1)^k \binom{r-1-k}{k} q^{1+k}(1+q)^{r-1-2k}
&=q\left(\frac{1-q^{n+1}}{1-q}\right).
\end{align*}
Now observe that $\displaystyle\sum_{i=1}^{r}q^i=\displaystyle\sum_{i=1}^{n+1}q^i=q\displaystyle\sum_{i=0}^{n}q^i=q\left(\frac{1-q^{n+1}}{1-q}\right)$.

Therefore $\displaystyle\sum_{k=0}^{\lfloor \frac{r-1}{2}\rfloor}(-1)^k \binom{r-1-k}{k} q^{1+k}(1+q)^{r-1-2k}=\displaystyle\sum_{i=1}^{r}q^i.$
\end{proof}

\begin{rmk} Suppose $r\ge 1$. If we define $F_r(t)=\sum_{k=0}^{\infty}\binom{r-1-k}{k}t^k$, then $F_r(1)$ is the $r^{th}$ Fibonacci number. So $F_r(t)$ is a t-analog of the Fibonacci numbers. Also notice if $t=\frac{-q}{(1+q)^2}$, then $q(1+q)^{r-1}F_r(t)$ is the sum we encountered in Proposition \ref{p7}.
\end{rmk}

\begin{proof}[Proof of Theorem \ref{t2}] By Lemma \ref{p5} and Propositions \ref{p6} and \ref{p7}, if $k=\ell(\sigma)$, then
\begin{align*}
m_q(\tilde\alpha,0)&=\displaystyle\sum_{\sigma\in W}\epsilon(\sigma)\wp_q(\sigma(\tilde\alpha+\rho)-\rho)\\
&=\displaystyle\sum_{\sigma\in\mathcal A(\tilde\alpha,0)}\epsilon(\sigma)\wp_q(\sigma(\tilde\alpha+\rho)-\rho)\\
&=\displaystyle\sum_{k=0}^{\lfloor\frac{r-1}{2}\rfloor}(-1)^{k}\binom{r-1-k}{k}q^{1+k}(1+q)^{r-1-2k}\\
&=q+q^2+q^3+\cdots+q^{r}.\end{align*}\end{proof}

\section{Non-zero weight spaces}

It is fundamental in Lie theory that the zero weight space is a Cartan subalgebra. If $\tilde\alpha$ is the highest root of $\mathfrak g$, then the non-zero weights of $L(\tilde\alpha)$, the adjoint representation of $\mathfrak g$, are the roots and have multiplicity 1. We visit this picture from our point of view in the case when $\mathfrak g=\mathfrak {sl}_{r+1}$. Let $r\ge 1$, and $n=r+1$.

\begin{thm}\label{nonzero}
If $\mu\in P_+(\mathfrak{sl}_{n})$ and $\mu\neq 0$, then $\mathcal A(\tilde\alpha,\mu)=\begin{cases} \{1\}& \text{if $\mu=\tilde\alpha$}\\ \emptyset &\text{otherwise.}\end{cases}$
\end{thm}

We begin by proving the following propositions.

\begin{prop}\label{p5.1} If $\tilde\alpha$ is the highest root of $\mathfrak {sl}_n$, then $\mathcal A(\tilde\alpha,\tilde\alpha)=\{1\}$.
\end{prop}
\begin{proof} It suffices to show that $\wp(\sigma(\tilde\alpha+\rho)-\rho-\tilde\alpha)>0$ if and only if $\sigma=1$.

$(\Rightarrow)$ Assume that $\sigma\in S_n$ such that $\wp(\sigma(\tilde\alpha+\rho)-\rho-\tilde\alpha)>0$. Proposition \ref{prop1} implies that $(\varpi_i,\sigma(\tilde\alpha+\rho)-\rho-\tilde\alpha)\in\mathbb N$, for all $1\le i\le r$. By (\ref{1.3}) we have that \begin{center}$\sigma(\tilde\alpha+\rho)-\rho-\tilde\alpha= (a_{\sigma^{-1}(1)}-n,\;a_{\sigma^{-1}(2)}-n+2,\ldots,\;a_{\sigma^{-1}(n-1)}-1,\;a_{\sigma^{-1}(n)}+1)$,\end{center} where $a_i$ is given by (\ref{relabel}).

Let $M=\{i\;|\:\sigma^{-1}(i)\neq i\}$. Suppose $M\neq \emptyset$, and let $j=min(M)$. Hence $\sigma^{-1}(j)=k$, for some integer $j<k\le n$. Since $\sigma^{-1}(i)=i$, for all $1\le i\le j-1$, and by definition of $a_i$, we have that $(\varpi_i,\sigma(\tilde\alpha+\rho)-\rho-\tilde\alpha)=0$, for all $1\le i\le j-1$.

Thus, \begin{center}$(\varpi_j,\sigma(\tilde\alpha+\rho)-\rho-\tilde\alpha)=a_{\sigma^{-1}(j)}-n+j=\begin{cases}j-n-1&\text{if $k=n$}\\ j-k&\text{if $j<k<n$.}\end{cases}$\end{center} In either case $(\varpi_j,\sigma(\tilde\alpha+\rho)-\rho-\tilde\alpha)<0$, a contradiction. Thus $M=\emptyset$, and $\sigma^{-1}(i)=i$, for all $1\le i\le n$. Therefore $\sigma=1$.

$(\Leftarrow)$ If $\sigma=1$, then $\wp(\sigma(\tilde\alpha+\rho)-\rho-\tilde\alpha)=\wp(0)=1>0$.
\end{proof}

\begin{prop}\label{t3} Let $\mu \in P_+(\mathfrak {sl}_n)$, and $\mu\neq 0$. Then there exists $\sigma\in S_n$ such that \\$\wp(\sigma(\tilde\alpha+\rho)-\rho-\mu)>0$ if and only if $\mu=\tilde\alpha$.
\end{prop}
\begin{proof}
$(\Rightarrow)$ If $\mu\in P_+(\mathfrak {sl}_n)$, then $\mu=(\mu_1,\mu_2,\ldots,\mu_n)$, for some $\mu_1,\ldots,\mu_n\in\mathbb Z$, satisfying $\mu_1\ge\mu_2\ge\cdots\ge\mu_n$ and $\sum_{i=1}^{n}\mu_i=0$. Assume $\mu\neq 0$, hence $(\mu_1,\ldots,\mu_n)\neq(0,\ldots,0)$. If $\mu_1<0$, then $\mu_i<0$, for all $2\le i\le n$ and $\sum_{i=1}^n \mu_i\neq 0$, a contradiction. Thus we may assume $\mu_1\ge 0$.

Now suppose there exists $\sigma\in S_n$ such that $\wp(\sigma(\tilde\alpha+\rho)-\rho-\mu)>0$. Proposition \ref{prop1} implies $(\varpi_i,\sigma(\tilde\alpha+\rho)-\rho-\mu)\in\mathbb N$, for all $1\le i\le n-1$. In particular,\\ $(\varpi_1,\sigma(\tilde\alpha+\rho)-\rho-\mu)=a_{\sigma^{-1}(1)}-n+1-\mu_1\ge 0$ if and only if $\mu_1\le a_{\sigma^{-1}(1)}-n+1$.

Observe that
\begin{center}$a_{\sigma^{-1}(1)}-n+1=\begin{cases}1&\text{if $\sigma^{-1}(1)=1$}\\ -n&\text{if $\sigma^{-1}(1)=n$}\\1-\sigma^{-1}(1)&\text{if $1<\sigma^{-1}(1)<n$.}\end{cases}$\end{center}
Then $(\varpi_1,\sigma(\tilde\alpha+\rho)-\rho-\mu)>0$ if and only if $\sigma^{-1}(1)=1$ and $\mu_1\le 1$. Since $\mu_1\ge 0$, we have that $\mu_1\in\{0,1\}$. If $\mu_1=0$, then $\sum_{i=1}^{n}\mu_i=0$ if and only if $\mu=0$, a contradiction. Therefore $\mu_1=1$.

Now observe that $(\varpi_{n-1},\sigma(\tilde\alpha+\rho)-\rho-\mu)=\sum_{i=1}^{n-1}(a_{\sigma^{-1}(i)}-n+i-\mu_i)
            =\mu_n-a_{\sigma^{-1}(n)}.$ Hence $(\varpi_{n-1},\sigma(\tilde\alpha+\rho)-\rho-\mu)\ge 0$ if and only if $\mu_n\ge a_{\sigma^{-1}(n)}$. If $1<\sigma^{-1}(n)<n$, then $a_{\sigma^{-1}(n)}=n-\sigma^{-1}(n)\ge 1$, and hence $\mu_n\ge 1$. Then $1=\mu_1\ge\mu_2\ge\cdots\ge\mu_n\ge 1$ implies that $\mu_i=1$, for all $1\le i\le n$ and so $\sum_{i=1}^{n}\mu_i\neq 0$, a contradiction. Therefore $\sigma^{-1}(n)=n$, and $\mu_n=-1$.

Observe that since $\mu_1=1$, $\mu_n=-1$, and $\mu_1\ge\mu_2\ge\cdots\ge\mu_n$, we have that $\mu_i\in\{1,0,-1\}$, for all $2\le i\le n-1$. If $\mu_2=-1$, then $\mu=(1,-1,\ldots,-1)$ and $\sum_{i=1}^n\mu_i\neq 0$, a contradiction. Suppose $\mu_2=1$. Since $\sigma^{-1}(1)=1$ and $\sigma^{-1}(n)=n$, we have that $1<\sigma^{-1}(2)<n$. Hence $(\varpi_2,\sigma(\tilde\alpha+\rho)-\rho-\mu)=a_{\sigma^{-1}(2)}-n+1=1-\sigma^{-1}(2)<0$, a contradiction. Thus $\mu_2=0$.

Notice that if $\mu_j=-1$, for some $2< j\le n-1$, then $\mu_i=-1$, for all $j<i\le n-1$. In which case $\sum_{i=1}^{n}\mu_i\neq 0$, giving rise to a contradiction. So $\mu_i=0$, for all $2\le i\le n-1$, and thus $\mu=(1,0,\ldots,0,-1)=\tilde\alpha$.

$(\Leftarrow)$ Follows from Proposition \ref{p5.1}.
\end{proof}

\begin{proof}[Proof of Theorem \ref{nonzero}] Follows from Proposition \ref{p5.1} and the contrapositive of Proposition \ref{t3}.
\end{proof}

The following corollary is fundamental in Lie theory. We give an alternate proof of this well known result by using Weyl alternation sets.

\begin{cor} If $\mu\in P(\mathfrak{sl}_{n})$, then $m(\tilde\alpha,\mu)=\begin{cases}r& \text{if $\mu=0$}\\ 1&\text{if $\mu\in\Phi$}\\0&\text{otherwise.}\end{cases}$
\end{cor}

\begin{proof}  By Proposition 3.1.20 in~\cite{GW}, if $\mu\in P(\mathfrak {sl}_n)$, then there exists $w\in W$ and $\xi\in P_+(\mathfrak {sl}_n)$ such that $w(\xi)=\mu$. Also by Proposition 3.2.27 in~\cite{GW} we know that weight multiplicities are invariant under $W.$ Thus it suffices to compute $m(\tilde\alpha,\mu)$ for $\mu\in P_{+}(\mathfrak{sl}_{n})$. By Corollary \ref{c1}, $m(\tilde\alpha,0)=r$. By Theorem \ref{nonzero} we know $\mathcal A(\tilde\alpha,\tilde\alpha)=\{1\}$, and $\mathcal A(\tilde\alpha,\mu)=\emptyset$ whenever $\mu\in P_+(\mathfrak {sl}_{n+1})-\{0,\tilde\alpha\}$. This implies that $m(\tilde\alpha,\tilde\alpha)=\wp(1(\tilde\alpha+\rho)-\rho-\tilde\alpha)=\wp(0)=1$, and that $m(\tilde\alpha,\mu)=0$ whenever $\mu\in P_+(\mathfrak {sl}_{n+1})-\{0,\tilde\alpha\}$.
\end{proof}

\end{document}